\documentclass[12pt, oneside, reqno]{amsart}
\usepackage{amsmath}
\usepackage{amssymb}
\usepackage{centernot}
\usepackage{color}
\usepackage[dvipsnames,x11names,svgnames]{xcolor}
\usepackage[colorlinks=true,linkcolor=NavyBlue,citecolor=DarkGreen, urlcolor=blue]{hyperref}
\usepackage{amsthm}
\usepackage{amscd}
\usepackage{geometry}
\usepackage{mathrsfs}
\usepackage{dsfont}
\usepackage{cite}
\usepackage{mathtools}

\usepackage{tikz}
\usetikzlibrary{positioning, shapes.geometric, fit, backgrounds}

\newcommand{\mc}{\mathcal}

\newcommand{\ms}{\mathscr}

\newcommand{\cG}{\mc{G}}
\newcommand{\cH}{\mc{H}}
\newcommand{\cM}{\mc{M}}

\newcommand{\cR}{\mc{R}}

\newcommand{\sF}{\ms{F}}
\newcommand{\sI}{\ms{I}}
\newcommand{\sJ}{\ms{J}}
\newcommand{\sR}{\ms{R}}
\newcommand{\sV}{\ms{V}}

\newcommand{\1}{\mathds{1}}

\newcommand{\N}{\mathbb{N}}
\newcommand{\C}{\mathbb{C}}

\newcommand{\F}{\mathbb{F}}
\newcommand{\Fq}{\mathbb{F}_q}

\newcommand{\symdiff}{\mathbin{\triangle}}

\newcommand{\rbinom}[3]{\genfrac{[}{]}{0pt}{}{#1}{#2,#3}}

\DeclareMathOperator{\supp}{supp}
\DeclareMathOperator{\Span}{Span}

\newcommand{\geqs}{\geqslant}
\newcommand{\leqs}{\leqslant}

\newtheorem{theorem}{Theorem}[section]

\newtheorem{corollary}[theorem]{Corollary}
\newtheorem{conjecture}[theorem]{Conjecture}
\newtheorem{proposition}[theorem]{Proposition}

\theoremstyle{definition}

\theoremstyle{remark}
\newtheorem{remark}[theorem]{Remark}

\begin{document}
 
\title[Note on M\"obius uncertainty for posets]{A note on the M\"obius uncertainty principle for posets}
\author{Anurag Sahay}
\address{Department of Mathematics, Purdue University, West Lafayette, IN 47907, USA}
\email{\href{mailto:anuragsahay@purdue.edu}{anuragsahay@purdue.edu}}

\begin{abstract}

We consider two generalizations of Pollack's uncertainty principle for M\"obius inversion \cite{pollack} to locally finite posets. 

The first generalization was previously studied by Goh \cite{goh}. Here, we provide a simplified sufficient criterion for the uncertainty principle to hold. We also provide a necessary criterion for the same which, in particular, disproves Goh's conjectural characterization of posets for which an uncertainty principle holds. Nevertheless, we prove that Goh's conjecture indeed holds when the poset forms a lattice.

The second generalization is new and applies to posets with reduced incidence algebras of a certain form. Here, we make some preliminary observations, including the fact that the uncertainty principle holds for the poset of finite subsets of natural numbers and the poset of finite dimensional subspaces of $\F_q^\infty$. Our proofs in these settings are quite different from the proof for the poset of natural numbers under divisibility. \end{abstract}

\maketitle

\section{Introduction}

In the short note \cite{pollack} (see also \cite{pollacksanna,sanna}), Pollack proves the following
\begin{theorem}[Pollack] \label{thm:pollack}
Let $f,g: \N \longrightarrow \C$ be arithmetic functions which are not identically zero such that
\begin{equation} \label{gis1cf} g(n) = \sum_{d \mid n} f(d).\end{equation}
Then, either $\supp(f)$ or $\supp(g)$ is infinite. Here, 
\[ \supp(f) = \{ n\in \N : f(n) \neq 0 \}, \]
is the support of $f$. 
\end{theorem}

Since \eqref{gis1cf} is equivalent to
\[ f(n) = \sum_{d\mid n} \mu\Big(\frac{n}{d}\Big) g(d) \]
by M\"obius inversion, Pollack calls this phenomenon ``an uncertainty principle for the M\"obius transform". 

Our goal in this note is to study this phenomenon in the context of more general locally finite posets; of these, the natural numbers under divisibility form a basic example. The reader unfamiliar with this theory may consult Rota's celebrated work \cite{rota} or Stanley's book \cite[Chapter 3]{stanley} before proceeding. The articles \cite{rota4fq,rota6reduced} shall also be relevant. 

To set terminology, we recall some notation. Let $(P, \leqs)$ be a poset which is locally finite. In other words, for any $x,y \in P$ with $x \leqs y$, the interval
\[ [x,y] := \{ z \in P : x \leqs z \leqs y \} \]
is always finite. For any ring $R$, one may then define the incidence algebra $\sI(P,R)$ of functions $f: P\times P \longrightarrow R$ with the property that
\[ f(x,y) = 0 \text{ whenever } x \centernot \leqs y. \]
In other words, $f$ is supported on tuples $(x,y)$ for which $[x,y]$ is an interval in $P$. The set $\sI(P,R)$ naturally has the structure of an $R$-algebra, where the multiplication $\ast$ is defined by
\[ (f\ast g)(x,y) = \sum_{x \leqs z \leqs y} f(x,z)g(z,y). \]
By the local finiteness property, this is a finite sum and hence, there are no issues of convergence. The zeta function of the poset $\zeta_P \in \sI(P,R)$ plays a distinguished role; it is defined by
\[ \zeta_P(x,y) = \begin{cases} 1 & \text{ if }x \leqs y, \\ 0 &\text{ otherwise.} \end{cases} \]
The zeta function is always invertible under $\ast$, and the inverse is the M\"obius function of the poset $\mu_P \in \sI(P,R)$ which can be defined recursively by
\begin{equation} \label{eq:mudef} \mu_P(x,y) = \begin{cases} 1 & \text{ if } x = y, \\ - \sum_{x \leqs z < y} \mu_P(x,z) & \text{ otherwise.} \end{cases} \end{equation}
Note that the sum is empty whenever $x \centernot \leqs y$. 

For any poset denoted by $P$, we will assume $P$ has a minimum element, which we denote by $\hat{0}$. Thus, $\hat{0} \leqs x$ for every $x \in P$. We will pay special attention to the subset $\sJ(P,R) \subset \sI(P,R)$ of functions $f:P \times P \longrightarrow R$ which are supported on intervals of the form $[\hat{0},x]$. Any such function can be identified with a function $f: P \longrightarrow R$ and vice-versa by setting $f(x) = f(\hat{0},x)$; we make this identification with no further comment.

We restrict attention to $R = \C$. The substance of our ensuing discussion will not be materially affected\footnote{Except in the proof of Proposition~\ref{prop:finsubsetN}, where $R$ must be an embedded subring of $\C$ and in the proof of Proposition~\ref{prop:finvectFq}, where $R$ must be a field of characteristic 0.} by assuming that $R$ is an arbitrary ring of characteristic $0$ instead. We will drop the notational dependence of $\sI$ and $\sJ$ on $R$ entirely and on $P$ when doing so causes no confusion.

With this context, we are ready to discuss Goh's generalization \cite{goh} of Pollack's work. For $f \in \sJ$, we define 
\[ \supp(f) = \{ x \in P : f(x) \neq 0 \} \]
to be the support of $f$ when viewed as a function on $P$. We say that $P$ has the M\"obius uncertainty property $\cM$ if for any nonzero functions $f,g \in \sJ(P)$ with
\[ g(z) = \sum_{x \leqs z} f(x) \text{ for every } z\in P, \]
we have that either $\supp(f)$ or $\supp(g)$ is infinite. Goh investigated when a poset $P$ has $\cM$, and proved a sufficient condition \cite[Theorem 2.1]{goh} for it. 

Building on his work, we will provide a simpler hypothesis than that of Goh \cite[Theorem 2.1]{goh} but which is nevertheless good enough for the applications \cite[\S 3]{goh} he was interested in. Further, we will show that a condition not terribly different from our sufficient condition is actually necessary for $P$ to have $\cM$. The sufficient condition is
\begin{theorem} \label{thm:sufficient}
Let $P$ be locally finite poset having $\hat{0}$ which has the property that  every nonempty finite subset $S \subset P$ contains an element $z \in S$ so that for infinitely many $x \in P$, one has $x \geqs z$ but $x \centernot \geqs y$ for any $y \in S\setminus \{z\}$. Then, $P$ has the M\"obius uncertainty property.
\end{theorem}

From this we obtain the following 
\begin{corollary} \label{cor:posetswithM}
The following posets have the M\"obius uncertainty property:
\begin{itemize}
\item The poset of natural numbers $\N$ under divisibility.
\item The poset $\sF(\N)$ of finite subsets of natural numbers under subset inclusion.
\item The poset $\sV(\Fq^\infty)$ of finite-dimensional subspaces of $\Fq^\infty$ under subspace inclusion.
\end{itemize}
\end{corollary}
The first of these is Theorem~\ref{thm:pollack}, the second is \cite[Proposition 3.2]{goh}, but the third appears to be novel. 

\begin{remark}
Theorem~\ref{thm:sufficient} is easy to apply; for example, it can be applied \emph{mutatis mutandis} to the poset under inclusion of integral ideals of the ring of integers $\mathcal{O}_K$ of a global field $K.$ 
\end{remark}

The necessary condition is
\begin{theorem} \label{thm:necessary}
Let $P$ be locally finite poset having $\hat{0}$ which has the M\"obius uncertainty property. Then, any subset $S \subset P$ satisfying the cardinality assumption $1 \leqs |S| \leqs 2$, contains an element $z \in S$ so that for infinitely many $x \in P$, one has that $x \geqs z$ but $x \centernot \geqs y$ for any $y \in S \setminus \{z\}$. 
\end{theorem}

For $k \in \N$, let us say that the poset $P$ has property $\cH_k$ if every subset $S \subset P$ with cardinality $|S| = k$ contains an element $z \in S$ so that for infinitely many $x \in P$ one has that $x \geqs z$ but $x \centernot \geqs y$ for any $y \in S \setminus \{z\}$. Further, let us say that $P$ has $\cH_\N$ if it has $\cH_k$ for every $k \in \N$. Thus, the above two theorems may be succinctly summarized as
\begin{equation} \label{eqn:mainthm} P \text{ has }\cH_\N \implies P \text{ has }\cM \implies P \text{ has }\cH_{1} \text{ and } \cH_2. \end{equation}

Goh speculates on \cite[p.~6]{goh} which of his conditions is truly necessary. In particular, he essentially conjectures the following:
\begin{conjecture}[Goh] \label{conj: goh}
A locally finite poset $P$ having $\hat{0}$ has the M\"obius uncertainty property if and only if for every $x \in P$, the set
\[ \{ y \in P : \mu_P(x,y) \neq 0 \} \]
is infinite.
\end{conjecture}

Writing $\cG$ for the M\"obius nonvanishing property, i.e., the hypothesis that for every $x \in P$ there are infinitely many $y$ so that $\mu_P(x,y) \neq 0$, we see that Goh is asserting that
\[ P\text{ has }\cM \iff P\text{ has }\cG. \]
Note that Goh proves in \cite[Proposition 4.1]{goh} that the forward implication holds. We disprove this conjecture by showing that the reverse implication does not. In particular, we can show
\begin{theorem} \label{thm:gohcounterexample}
There exists a locally finite poset which has the property $\cG$ but does not have the property $\cH_2$. In particular, since $\cM \implies \cH_2$, there exists a locally finite poset which has the M\"obius nonvanishing property but does not have the M\"obius uncertainty property.
\end{theorem}

\begin{remark} \label{rem:goh}
Conjecture~\ref{conj: goh} as stated above is a special case of Conjecture 6 in an earlier arXiv version (\href{https://arxiv.org/abs/2302.02466v2}{arXiv: 2302.02466v2}) of \cite{goh}. In the published version, Goh is more restrictive, requiring that $P$ be a join semi-lattice\footnote{A poset $P$ is a join semi-lattice if every finite and nonempty subset $S\subset P$ has a least upper bound in $P$. If every finite and nonempty subset has both a least upper bound and a greatest lower bound, then $P$ is simply called a lattice.}. Goh has indicated to me in private communication that this change was made for aesthetic reasons; he did not have a counterexample in mind when making this change.
\end{remark}

Let us state (a special case of) the published version of Goh's conjecture \cite[Conjecture 4.3]{goh} explicitly. 

\begin{conjecture}[Goh] \label{conj: gohl}
A locally finite join semi-lattice $L$ having $\hat{0}$ has the M\"obius uncertainty property if and only if for every $x \in L$, the set
\[ \{ y \in L : \mu_L(x,y) \neq 0 \} \]
is infinite.
\end{conjecture}

Our counterexample is manifestly not a semi-lattice, so it does not disprove this conjecture. To the contrary, we have the following 

\begin{theorem} \label{thm:gohtrue}
If $L$ is a locally finite lattice having $\hat{0}$, then the following are equivalent:
\begin{itemize}
\item The lattice $L$ has the M\"obius uncertainty property $\cM$.
\item The lattice $L$ has the M\"obius nonvanishing property $\cG$.
\item The lattice $L$ has the property $\cH_\N$.
\end{itemize}
Thus,
\[ L \text{ has } \cM \iff L \text{ has } \cG \iff L \text{ has } \cH_\N. \]
\end{theorem}

It may seem that the hypotheses of this theorem are stronger than those of Conjecture~\ref{conj: gohl}, but in fact any locally finite join semi-lattice having $\hat{0}$ is a lattice; see \S\ref{sec: gohl}. Thus, this settles Conjecture~\ref{conj: gohl} affirmatively. 

In light of \eqref{eqn:mainthm}, Theorem~\ref{thm:gohcounterexample}, and Theorem~\ref{thm:gohtrue}, I believed for a while that $\cM \iff \cH_\N$ may be the correct structural characterisation of when the M\"obius uncertainty principle holds for general posets. Unfortunately, this hope was dashed when I found an example showing that the latter implication in \eqref{eqn:mainthm} is closer to being tight. We will show the following
\begin{theorem} \label{thm:P3counterexample}
There exists a locally finite poset which has the M\"obius uncertainty property $\cM$ but does not have the property $\cH_3$. Thus, $\cM \centernot\implies \cH_3$ and hence $\cM \centernot\implies \cH_\N$. 
\end{theorem}

This leaves open the possibility that
\[ P \text{ has } \cM \iff P \text{ has } \cH_1 \text{ and } \cH_2, \]
for general posets, but I consider this quite unlikely. 

At this point, we will introduce another way to generalize Theorem~\ref{thm:pollack} using the formalism of reduced incidence algebras from \cite{rota6reduced}. We call a map $t: P \times P \longrightarrow \N \cup \{0\}$ a \emph{typing} on $P$ if it satisfies the following conditions:
\begin{itemize}
\item For $x,y \in P$, $t(x,y) \neq 0$ if and only if $[x,y]$ is an interval in $P$.
\item For any functions $f_1, f_2 \in \sI(P)$ such that $f_j(x,y)$ depends only on $t(x,y)$, we have that $(f_1 \ast f_2)(x,y)$ also only depends on $t(x,y)$.
\end{itemize}
We need only specify $t(x,y)$ when $x \leqs y$, as $t(x,y) = 0$ otherwise. Given a typing on $P$, we call $t(x,y)$ the \emph{type} of $[x,y]$ and we call the set of functions in $\sI$ whose value on $[x,y]$ depends only on its type the reduced incidence algebra of $(P,t)$, which we denote by $\sR$. This is a slightly less general notion than that used in \cite{rota6reduced}, but to see the connection, $[x,y] \sim [u,v]$ iff $t(x,y) = t(u,v)$ is an order compatible equivalence relation on the intervals of $P$ in the sense of \cite[Definition 4.1]{rota6reduced}. The reduced incidence algebra $\sR$ is also an $R$-algebra under the convolution operator $\ast$.

If the map $t$ surjects onto the natural numbers, functions $f \in \sR$ can be thought of as functions on the natural numbers, by defining $f(n) = f(x,y)$ where $[x,y]$ is any interval of type $n$. From now on, we assume $t$ is surjective and make this association. In this notation, one has that
\[ (f\ast g)(n) = \sum_{d, k \in \N} \rbinom{n}{d}{k} f(d) g(k), \]
where $\rbinom{n}{d}{k}$ are the incidence coefficients (or structure coefficients) of $\sR$, which is the number of distinct elements $z$ in an interval $[x,y]$ of type $n$ such that $[x,z]$ is of type $d$ and $[z,y]$ is of type $k$. 

The zeta function satisfies $\zeta_P \in \sR$, and in fact, $\zeta_P(n) = 1$ for every $n \in \N$. Thus, for $f,g \in \sR$ we consider it a M\"obius pair if 
\[ g = \zeta_P \ast f \iff f = \mu_P \ast g. \]
Writing the first equality out as functions on integers, we have that
\begin{equation} \label{eq:redg=f} g(n) = \sum_{d, k \in \N} \rbinom{n}{d}{k} f(k). \end{equation}
Our generalization of the uncertainty principle is now as follows: we say that a poset $P$ with a surjective typing $t$ has the reduced M\"obius uncertainty property $\cR$ if for every $f,g \in \sR$ satisfying \eqref{eq:redg=f} at most one of $\supp_\sR(f)$ or $\supp_\sR(g)$ is finite, where
\[ \supp_\sR(f) = \{ n \in \N : f(n) \neq 0 \} \]
is the support of $f$ when viewed as a function of $\N$. In my view, this is a slightly more natural generalization than the previous one; however we will be only able to make preliminary observations about it. We end the introduction by recording these in the following 
\begin{theorem} \label{thm:reduced}
The following posets have the property $\cR$:
\begin{itemize}
\item The poset of natural numbers $\N$ under divisibility with the typing $t(n,m) = m/n$.
\item The poset $\sF(\N)$ of finite subsets of natural numbers under subset inclusion with the typing $t(S,T) = |T| - |S|+1$.
\item The poset $\sV(\Fq^\infty)$ of finite-dimensional subspaces of $\Fq^\infty$ under subspace inclusion with the typing $t(U,V) = \dim V - \dim U + 1$.
\end{itemize}
On the other hand, the poset of natural numbers $\N$ under the usual order and the typing $t(n,m) = m - n + 1$ does not have the property $\cR$. 
\end{theorem}
\subsection*{Acknowledgments} The author is partially supported by Trevor Wooley's start-up funding at Purdue University and by the AMS-Simons Travel Grant. I am grateful to Paul Pollack for a stimulating discussion. I am also thankful to Marcel Goh, Ofir Gorodetsky, Akshat Mudgal, and Brad Rodgers for their encouragement and for their comments on this manuscript.

\section{The sufficient condition}

In this section, we prove Theorem~\ref{thm:sufficient} and Corollary~\ref{cor:posetswithM}. Before doing so, we present the key idea in the context of Theorem~\ref{thm:pollack}. Thus, suppose that $f,g:\N \longrightarrow \C$ are arithmetic functions where $f$ is finitely supported but it is not identically zero. To prove Theorem~\ref{thm:pollack}, it suffices to show that $g$ is not finitely supported. Writing $S = \supp(f)$, we find that
\begin{equation} \label{eq:g=find} g(n) = \sum_{d \mid n} f(d) = \sum_{d \in S} f(d) \1_{d \mid n}. \end{equation}
Now, as $n$ runs through integers divisible by exactly one $d \in S$, we find that 
\[ g(n) = f(d) \neq 0. \]
However, this set is clearly infinite since $S$ is not the empty set.

\begin{remark}
The assertion that there are infinitely many integers divisible by exactly one $d \in S$ uses implicitly the infinitude of primes. However, this may be avoided if an aficionado of \cite{pollack} is so inclined: instead, observe that the right hand side of \eqref{eq:g=find} is a linear combination of periodic functions of $n$, and hence must itself be periodic. Since $g$ is periodic and is not identically zero, this implies that $\supp(g)$ is infinite. However, this modification does not appear to generalize to arbitrary posets with the property $\cH_\N$.
\end{remark}

We now implement this proof in the general setting.
\begin{proof}[Proof of Theorem~\ref{thm:sufficient}]
Let $P$ be the poset having the property $\cH_\N$. To prove that $P$ has $\cM$, we need to show that for every nonzero $f,g: P \longrightarrow \C$ satisfying
\begin{equation} \label{eqn:g=fsum} g(x) = \sum_{z \leqs x} f(z), \end{equation}
at most one of $\supp(f)$ and $\supp(g)$ is finite. To this end, suppose $\supp(f)$ is finite but not empty and denote it by $S$. By the hypothesis $\cH_\N$, it follows that there exists an element $z \in S$ such that the set
\[ \big\{ x \in P : x \geqs z \text{ but } x \centernot \geqs y \text{ for any } y \in S \setminus \{z\} \big\} \]
is infinite. However, for $x$ in this set, we find from \eqref{eqn:g=fsum} that
\[ g(x) = f(z) \neq 0. \]
This establishes the claim, as we have shown that $\supp(g)$ contains an infinite set.
\end{proof}

We now establish the corollary.

\begin{proof}[Proof of Corollary~\ref{cor:posetswithM}]
By Theorem~\ref{thm:sufficient}, it suffices to show that the posets in question have the property $\cH_\N$. 

For the poset of $\N$ under divisibility, this is clear and has already been used at the beginning of this section. 

For the poset $\sF(\N)$, let $S$ be a finite and nonempty set of elements in $\sF(\N)$. Let $U$ be a minimal element in $S$ with respect to subset ordering. We will construct an infinite sequence $\{ V_j \}$ of finite subsets of $\N$ which contain $U$ but contain no other element of $S$. Since $S$ is finite, the union of all its elements is also a finite subset of $\N$. Thus, in particular, there are infinitely many elements in $\N$ which do not appear in any element of $S$. Arranging these in sequence $\{ v_j\}_{j\in \N}$, we define $V_j = U \cup \{ v_j \}$. Thus, clearly $U \subseteq V_j$. On the other hand, for $W \in S$ such that $W \subseteq V_j$, since $v_j \notin W$ by construction, we get that $W \subseteq U$. This contradicts the minimality of $U$ unless $W = U$, as desired.

For the poset $\sV(\Fq^\infty)$, a similar argument works. In fact, this may be seen as the $q$-analogue of the previous argument. If $S$ is a collection of finite-dimensional vector subspaces of $\Fq^\infty$, we select $U$ to be a minimal element of $S$ with respect to subspace ordering and take $V_j = \Span\{ U, v_j\}$ where $v_j$ runs through an infinite sequence of linearly independent vectors that do not lie in $\Span S$.
 
\end{proof}

\section{The necessary condition}

In this section, we prove Theorem~\ref{thm:necessary}. 

\begin{proof}[Proof of Theorem~\ref{thm:necessary}]

To prove this theorem, it suffices to show that if $P$ is a locally finite poset which lacks either of the properties $\cH_1$ or $\cH_2$, then it must lack the M\"obius uncertainty property. 

Suppose first that $P$ lacks $\cH_1$. Thus, there is an element $z \in P$ with the property that there are only finitely many elements in the up-set
\[ U(z) = [z,\infty) = \{ y \in P : y \geqs z \}. \]
Without losing generality, we may assume that $U(z) = \{z\}$ by replacing $z$ with a maximal element in $U(z)$. Now, consider the function
\[ f(x) = \1_z(x) = \begin{cases} 1 & \text{ if } x = z, \\ 0 & \text { otherwise.} \end{cases} \]
For this choice of $f$,
\[ g(x) = \sum_{y \leqs x} f(y) = \sum_{y \leqs x} \1_z(y) = \1_{z\leqs x} = \1_z(x) . \]
In the last step, we used $U(z) = \{z\}$. In particular, $\supp(f) = \supp(g) = \{z\}$, whence $P$ does not have $\cM$.

Now suppose that $P$ lacks $\cH_2$. Thus, there is a two-element set $\{z_1,z_2\}$ in $P$ such that the symmetric difference of their up-sets
\[ U(z_1) \symdiff U(z_2) := \Big(U(z_1) \setminus U(z_2)\Big) \cup \Big(U(z_2) \setminus U(z_1)\Big) \]
is finite. Now, define 
\begin{equation} \label{eq:falt} f(x) = \1_{z_2}(x) - \1_{z_1}(x)=\begin{cases} (-1)^j & \text{ if } x = z_j \text{ for } j \in \{1,2\}, \\ 0 & \text{ otherwise.} \end{cases} \end{equation}
Since $\supp(f) = \{z_1,z_2\}$ is finite, it now suffices to show that
\[ \supp(g) \subseteq U(z_1) \symdiff U(z_2)  \]
as the latter is finite. To this end, suppose that $x$ does not belong to the latter set. First, let us deal with the case $x \notin U(z_j)$ for $j = 1,2$. In this case,
\[ g(x) = \sum_{y\leqs x} f(y) = 0, \]
since we have assumed that $x$ does not dominate any element in the support of $f$. Second, consider the case $x \in U(z_j)$ for $j = 1,2$. In this case,

\[ g(x) = \sum_{y\leqs x} f(y) = f(z_1) + f(z_2) = -1 + 1 = 0, \]
since $x$ dominates both elements in the support of $f$. Thus, we see that the support of $g$ must be contained in the symmetric difference of $U(z_1)$ and $U(z_2)$ as claimed. 

\end{proof}

\section{Proof of Theorem~\ref{thm:gohtrue}} \label{sec: gohl}

We begin this section by pointing out that a locally finite join semi-lattice $L$ with $\hat{0}$ must necessarily be a lattice. It suffices to show binary meets $y \wedge z$ exist in $L$. Note that the down-set
\[ D(z) = [\hat{0},z] = \{ x \in L : x \leqs z \} \]
is finite and contains $\hat{0}$. Thus, $D(y) \cap D(z)$ is finite and non-empty, and hence has a join. It is a standard exercise to check that this join is in fact the meet $y \wedge z$. 

It follows that if $x,y,z \in L$, with $x \leqs y, z$ then
\begin{equation} \label{eq:princdown} [x,y] \cap [x,z] = [x,y \wedge z]. \end{equation}
This will be used crucially in the proof of Theorem~\ref{thm:gohtrue}, which we now present.

\begin{proof}[Proof of Theorem~\ref{thm:gohtrue}]
The implications 
\[ L \text{ has } \cH_\N \implies L \text{ has } \cM \implies L \text{ has } \cG \]
are already known, the former following from Theorem~\ref{thm:sufficient} and the latter from \cite[Proposition 4.1]{goh}. Thus, it suffices to prove that
\[ L \text{ has } \cG \implies L \text{ has } \cH_\N. \]
We proceed by contrapositive. If $L$ lacks $\cH_\N$, then there exists a nonempty and finite subset $S \subset L$ such that for every $x \in S$, the difference of up-sets
\[ U(x) \setminus \bigcup_{\substack{v \in S\\v \neq x}} U(v) \]
is finite. Choosing $x$ to be the minimal element in $S$, we can ensure this set contains $x$. Then, replacing $x$ with a maximal element in the above set, we can assume without losing generality that
\[ \{x\} =  U(x) \setminus \bigcup_{\substack{v \in S\\v \neq x}} U(v) \]
Further, since $S$ is finite, $S$ has a supremum, which we shall denote by $z$. Setting $E = [x,z]$, we shall prove that for this choice of $x$ and $E$, we have that
\begin{equation} \label{eq:muvan} \mu_L(x,y) = 0 \text{ for any } y \in U(x) \setminus E. \end{equation}
Since $\mu_L(x,y) = 0$ whenever $y \notin U(x)$, it follows that
\[ \{ y \in L : \mu_L(x,y) \neq 0 \} \subseteq E \]
and hence it is finite since $E = [x,z]$ is finite. Thus, $L$ lacks the M\"obius nonvanishing property $\cG$.

It remains to prove \eqref{eq:muvan}. We proceed by induction on the size of the set $[x,y] \setminus E$, the base case being when $[x,y] \setminus E = \{y\}$. By \eqref{eq:mudef}, we have that
\[ \mu_L(x,y) = - \sum_{x \leqs u < y} \mu_L(x,u). \]
In the base case, any $u$ that appears in this sum is an element of $E$. In fact,
\[ [x,y) = [x,y] \cap E = [x,y] \cap [x,z] = [x,y_0], \]
where $y_0 = y \wedge z$. Thus,
\[ \mu_L(x,y) = - \sum_{x \leqs u \leqs y_0} \mu_L(x,u) = 0 \]
where the last equality follows from the recursion \eqref{eq:mudef} defining $\mu_L(x,y_0)$. This proves the base case.

For the inductive step, we have that
\begin{equation*}
\begin{split}
\mu_L(x,y) & = - \sum_{x \leqs u < y} \mu_L(x,u) \\ & = - \sum_{\substack{x \leqs u < y\\u \notin E}} \mu_L(x,u) - \sum_{\substack{x \leqs u < y\\ u \in E}} \mu_L(x,u). 
\end{split}
\end{equation*}
By the inductive hypothesis, $\mu_L(x,u) = 0$ for any $u \notin E$, since $[x,u] \setminus E$ clearly has fewer elements than $[x,y] \setminus E$. Hence, the first sum here vanishes. On the other hand, the same argument as in the base case will deliver that the second sum also vanishes. Thus, $\mu_L(x,y) = 0$ closing the induction and completing the proof.
\end{proof}
\section{Counterexamples} \label{sec:counterexample}

In this section, we provide the examples that prove Theorem~\ref{thm:gohcounterexample} and Theorem~\ref{thm:P3counterexample}. We will need a modicum more from the theory of posets. If $P$ and $Q$ are locally finite posets, we endow $P \times Q$ with the ordering 
\[ (p_1,q_1) \leqs (p_2,q_2) \text{ in } P\times Q \iff p_1 \leqs p_2 \text { in } P \text{ and } q_1 \leqs q_2 \text{ in } Q. \]
This turns $P \times Q$ into a locally finite poset as well. 

\begin{proof}[Proof of Theorem~\ref{thm:gohcounterexample}]
We will construct the desired poset in steps. To begin with, denote by $A$ the antichain poset on $\N$ (where each pair of nonequal elements is incomparable) and by $D$ the standard divisibility poset on $\N$. Set $P_1 = A \times D$. 

To $P_1$, we will now attach three elements $\{u, z_1, z_2 \}$ to obtain $P_2$. The new poset $P_2$ has all the inequalities as $P_1$ and also has the inequalities
\[ u < z_1, z_2 < p \]
for every element $p \in P_1$. There are no other nontrivial inequalities in $P_2$, and hence in particular $z_1$ and $z_2$ are incomparable. 

Finally, to $P_2$, we attach a copy of $D \setminus \{1\}$ (which we will denote $D_0$) to obtain $P$. In $P$, we include all the inequalities implicit in $P_2$ and $D_0$ and we further add the inequalities
\[ u < d \]
for every $d \in D_0$. No other inequalities are added. In particular, we see that $D_0 \cup \{u\}$ is naturally isomorphic to $D$, under the map which is the inclusion on $D_0$ and maps $u$ to $1$. 

Figure~\ref{fig:gohcounterexample} provides a schematic description of the poset $P$.
\begin{figure}[ht]
\centering
\caption{Schematic of the Hasse diagram for $P$.}
\label{fig:gohcounterexample}
\begin{tikzpicture}[
  point/.style={circle, draw, inner sep=1.5pt},
  big/.style={ellipse, draw, minimum width=3cm, minimum height=1cm, align=center},
  d0/.style={ellipse, draw, minimum width=4cm, minimum height=2cm, align=center}
]

\node[big] (p1) at (0,5) {$P_1 = A \times D$};

\node[point] (z1) at (-1.2,3.5) {$z_1$};
\node[point] (z2) at (1.2,3.5) {$z_2$};

\node[point] (u) at (0,1.5) {$u$};

\draw (u) -- (z1);
\draw (u) -- (z2);
\draw (z1) -- (p1);
\draw (z2) -- (p1);

\node[draw, thick, rectangle, inner sep=6pt, fit=(p1) (z1) (z2) (u), label=above right:$P_2$] {};

\node[point] (d2) at (4,3.5) {$2$};
\node[point] (d3) at (5,3.5) {$3$};
\node[point] (d5) at (6,3.5) {$5$};
\node[point] (d4) at (4,4.5) {$4$};
\node[point] (d6) at (5,4.5) {$6$};

\draw (d2) -- (d4); 
\draw (d2) -- (d6); 
\draw (d3) -- (d6); 

\draw (u) -- (d2);
\draw (u) -- (d3);
\draw (u) -- (d5);

\node[d0, thick, fit=(d2) (d3) (d4) (d5) (d6), label=above right:$D_0$] {};

\end{tikzpicture}
\end{figure}
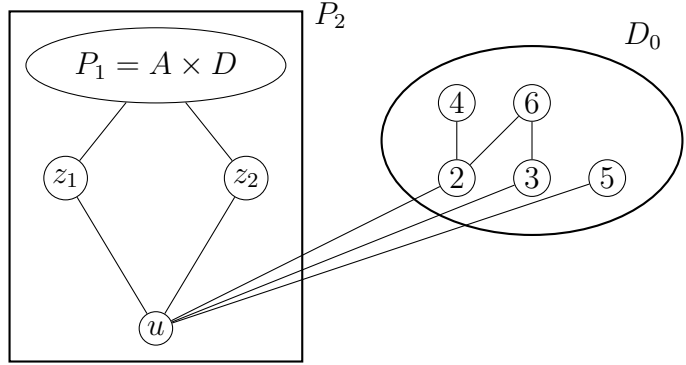

We now show that $P$ has the property $\cG$. To this end, let $x \in P$. If $x$ is contained in a copy of $D$ (say, if it is $u$, an element of $D_0$, or an element of $P_1$), then we obtain immediately that 
\begin{equation} \label{eq:setofmu!=0} \{ y \in P : \mu_P(x,y) \neq 0 \} \end{equation}
is infinite. This is because we may choose $y$ to lie in the same copy of $D$ as $x$; since $\mu_P(x,y)$ depends only on the interval $[x,y]$, it follows that $\mu_P(x,y) = \mu_D(x,y)$ where in the latter expression we are considering $[x,y]$ as the isomorphic interval in $D$. However, $\mu_D(x,y) \neq 0$ for infinitely many $y$, settling this case.

It remains to consider the case $x = z_j$ for some $j$. In this case, define $m_\ell \in P_1 = A \times D$ to be the tuple $(\ell,1)$. By construction, we have that $x \leqs m_\ell$ for every $\ell$. On the other hand, it is easy to see that the interval $(x,m_\ell)$ is empty; if there were a tuple $(s,t) \in P_1$ satisfying $(s,t) \leqs m_\ell = (\ell,1)$ then by the order in $A \times D$, it must be the case that $s = \ell$ and $t$ divides $1$ and hence $t = 1$. However, this implies that
\[ \mu_P(x,m_\ell) = -\sum_{x \leqs y < m_\ell} \mu_P(x,y) = - \mu_P(x,x) =  -1  \]
for every $\ell \in \N$, thus proving that the set \eqref{eq:setofmu!=0} has infinitely many elements. 

Despite this, it is immediately apparent that $P$ does not have $\cH_2$, since manifestly,
\[ U(z_1) \symdiff U(z_2) = \{ z_1, z_2 \}, \]
and hence $\{z_1,z_2\}$ is a witness to $\cH_2$ failing. In particular, we see that the function $f$ defined by \eqref{eq:falt} furnishes a counterexample to the M\"obius uncertainty principle $\cM$.
\end{proof}

A few remarks are in order. The idea of attaching $z_1$ and $z_2$ below $A \times D$ is the crucial one. The inclusion of $u$ is only to ensure that $P$ has a minimal element $\hat{0}$, and the inclusion of $D_0$ is to ensure that the set \eqref{eq:setofmu!=0} remains infinite for $x = u$. The reader may wish to check that this counterexample is not a lattice, since $z_1 \wedge z_2$ does not exist.

We now turn to proving Theorem~\ref{thm:P3counterexample}. 

\begin{proof}[Proof of Theorem~\ref{thm:P3counterexample}]
We again describe the construction in steps. To begin with, keeping $A$ and $D$ as in the proof of the previous theorem, define $Q_1 = A \times D$. Now, take three copies of $Q_1$, which we label $Q_{a}$, $Q_{b}$, and $Q_{c}$ together with three elements $\{a,b,c\}$. We keep the inequalities implicit in $Q_x$ for $x \in \{a,b,c\}$ and for $x, y \in \{a,b,c\}$ with $x \neq y$, we add the inequality
\[ x < q \]
whenever $q \in Q_y$. In other words, $Q_y$ dominates every element from $\{a,b,c\}$ except $y$ itself. We call the resulting poset $Q_2$. Finally, we obtain the desired poset $Q$ by adding a minimal element $u$ to $Q_2$ together with an attached copy of $D_0$ as in the previous theorem. 

Figure~\ref{fig:p3counterexample} hopefully provides an illuminating diagram for $Q$.

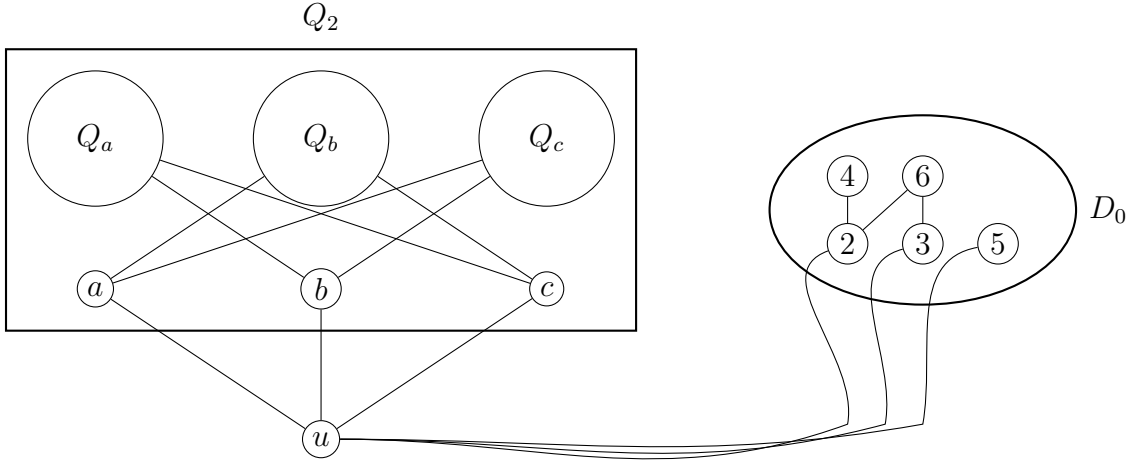
\begin{figure}[ht]
\centering
\caption{Schematic of the Hasse diagram for $Q$.}
\label{fig:p3counterexample}

\begin{tikzpicture}[
scale=1,
every node/.style={draw, circle, inner sep=2pt},
Qcircle/.style={draw, circle, minimum size=1.8cm},
box/.style={draw, rectangle, inner sep=8pt},
Doval/.style={draw, ellipse, inner sep=4pt}
]

\node (u) at (0,-0.4) {$u$};

\node (a) at (-3,1.6) {$a$};
\node (b) at (0,1.6) {$b$};
\node (c) at (3,1.6) {$c$};

\draw (u) -- (a);
\draw (u) -- (b);
\draw (u) -- (c);

\node[Qcircle] (Qa) at (-3,3.6) {$Q_a$};
\node[Qcircle] (Qb) at (0,3.6) {$Q_b$};
\node[Qcircle] (Qc) at (3,3.6) {$Q_c$};

\draw (a) -- (Qb);
\draw (a) -- (Qc);
\draw (b) -- (Qa);
\draw (b) -- (Qc);
\draw (c) -- (Qa);
\draw (c) -- (Qb);

\node (d2) at (7,2.2) {$2$};
\node (d3) at (8,2.2) {$3$};
\node (d5) at (9,2.2) {$5$};
\node (d4) at (7,3.1) {$4$};
\node (d6) at (8,3.1) {$6$};

\draw (d2) -- (d4);
\draw (d2) -- (d6);
\draw (d3) -- (d6);

\draw (u) to[out=0,in=-160] (7,-0.2) to[out=80,in=-160] (d2);
\draw (u) to[out=0,in=-165] (7.5,-0.2) to[out=80,in=-165] (d3);
\draw (u) to[out=0,in=-170] (8,-0.2) to[out=80,in=-170] (d5);

\node[Doval, thick, fit={(6.7,1.9) (9.3,3.4)}, label=right:$D_0$] {};

\node[box, thick, fit={(Qa) (Qb) (Qc) (a) (b) (c)}, label=above:$Q_2$] {};

\end{tikzpicture}

\end{figure}

Observe that $Q$ as constructed has the properties $\cH_1$ and $\cH_2$ but it does not have the property $\cH_3$, since for $S = \{a,b,c\}$, we see that there is no element in $Q$ that dominates exactly one element in $S$. However, as we will now show, $Q$ still has the M\"obius uncertainty property $\cM$. 

To see this, let $f,g \in \sJ(Q)$ be such that $f$ is finitely supported but nonzero and 
\[ g(x) = \sum_{z \leqs x} f(z). \]
Writing $D_1 = D_0 \cup \{u\}$, it is easy to see that the restriction $g\big|_{D_1}$ depends only on $f\big|_{D_1}$. In particular, if the support of $f$ intersects $D_1$, then $g$ is infinitely supported since $D_1$ is isomorphic to $D$. Thus, from now on, we can assume that $\supp(f)$ does not intersect $D_1$ and hence also that $f(u) = 0$. 

Now, suppose that $f(a) + f(b) \neq 0$. Consider the element $m_\ell \in Q_c$, which is defined by $m_\ell = (\ell,1)$ when $Q_c$ is regarded as $A \times D$. Since $f$ is finitely supported, there are infinitely many $\ell \in \N$ such that $f(m_\ell) = 0$. Since $m_\ell$ is minimal in $Q_c$, we find that
\[ g(m_\ell) = \sum_{z \leqs m_\ell} f(z) = f(a) + f(b) \neq 0 \]
whenever $f(m_\ell) = f(u) = 0$. Since the latter happens for infinitely many $\ell$, we see that $g$ has infinite support in this case.

Thus, we may now suppose that $f(a) + f(b) = 0$. Due to the $S_3$-symmetry on the set $\{a,b,c\}$ inherent in the construction of $Q$, we may also suppose that $f(a) + f(c) = 0$ and $f(b) + f(c) = 0$. From all of this, it follows that
\[ f(a) = f(b) = f(c) = 0. \]
Now, since $f$ has some support, we may assume (by using the $S_3$-symmetry if necessary) that $f(z) \neq 0$ for some $z \in Q_a$. However, $Q_a$ is composed of disjoint copies of $D$, and the restriction of $g$ to any such $D$ can depend at most on $f$ restricted to this $D$ (this is because $f(b) = f(c) = f(u) = 0$). We may thus conclude that $g$ is infinitely supported with the same reasoning as with $D_1$ above.

This covers all the cases, concluding the proof. 
\end{proof}

\section{The uncertainty principle for reduced incidence algebras}

We begin this section by remarking that typings specified in Theorem~\ref{thm:reduced} are the natural typings on these posets. The reduced incidence algebras in these cases are naturally isomorphic to algebras of formal generating functions (see \cite{rota6reduced}):
\begin{itemize}
\item For $\N$ under division, $f \mapsto \sum f(n)/n^s$ identifies $\sR$ with the algebra of Dirichlet series. The structure coefficients are
\[ \rbinom{n}{d}{k} = \1_{n = dk}. \]
\item For $\N$ under linear ordering, $f \mapsto \sum f(n)x^{n-1}$ identifies $\sR$ with the algebra of power series. The structure coefficients are
\[ \rbinom{n}{d}{k} = \1_{n = d + k - 1}. \]
\item For $\sF(\N)$, $f \mapsto \sum f(n) x^{n-1}/(n-1)!$ identifies $\sR$ with the algebra of exponential power series. The structure coefficients are
\[ \rbinom{n}{d}{k} = \1_{n = d + k - 1} \binom{n-1}{k-1}. \]
\item For $\sV(\Fq^\infty)$, $f \mapsto \sum f(n) x^{n-1}/[n-1]!_q$ identifies $\sR$ with the algebra of Eulerian power series (this being the $q$-analogue of the previous example, albeit with $q$ restricted to prime powers). The structure coefficients are
\[ \rbinom{n}{d}{k} = \1_{n = d + k-1} \binom{n-1}{k-1}_q. \]
\end{itemize}

Let us begin the proof of Theorem~\ref{thm:reduced} in earnest. The case of $\N$ under division is contained already in Theorem~\ref{thm:pollack}. The case of $\N$ under the usual ordering is easily settled, as in this case, 
\[ \mu_P(n) = \begin{cases} 1 & \text{ if } n = 1, \\ -1 & \text{ if } n = 2, \\ 0 & \text{ otherwise.} \end{cases} \]
Clearly $\zeta_P \ast \mu_P$ and $\mu_P$ are both finitely supported in this case. 

We give the proofs of the other two claims the status of self-contained propositions.

\begin{proposition} \label{prop:finsubsetN}
Let $f,g:\N \longrightarrow \C$ be functions which are not identically zero such that 
\[ g(n) = \sum_{k=1}^{n} \binom{n-1}{k-1} f(k). \]
Then, either $\supp(f)$ or $\supp(g)$ is infinite. 
\end{proposition}

\begin{proof}
Suppose $f$ is finitely supported, with $S = \supp(f)$ and let $k_1$ be the largest integer in $S$. This implies that if $n> k_1$,
\[ g(n) = \sum_{k \in S} \binom{n-1}{k-1} f(k). \]
However, since $\binom{n}{k} \asymp_{k} n^k$, we find that
\begin{align*} g(n) & = \binom{n-1}{k_1-1} f(k_1) + \sum_{k < k_1} \binom{n-1}{k-1} f(k) \\ & = \binom{n-1}{k_1-1} f(k_1) + O(n^{k_1 - 2}) \\&  \asymp n^{k_1-1}. \end{align*}
Implicit constants in the asymptotic notation here may depend on $f$. In any case, this implies that $g(n) \neq 0$ for sufficiently large $n$, proving the proposition.
\end{proof}

\begin{proposition} \label{prop:finvectFq}
Let $f,g:\N \longrightarrow \C$ be functions which are not identically zero such that 
\[ g(n) = \sum_{k=1}^n \binom{n-1}{k-1}_q f(k). \]
Then, either $\supp(f)$ or $\supp(g)$ is infinite. Here, 
\[ \binom{n}{k}_q = \frac{[n]!_q}{[k]!_q [n-k]!_q} \]
is the $q$-binomial and
\[ [n]!_q = \prod_{d=1}^n [d] = \prod_{d=1}^n \frac{q^d-1}{q-1} \]
is the $q$-factorial.
\end{proposition}

\begin{proof}
For $k \geqs 0$, we define the polynomials
\[ P_k(u) = \prod_{d = 1}^k \Big(\frac{1-q^{1-d}u}{1-q^d}\Big). \]
For different $k$, these are of differing degrees, so they must linearly independent over $\C$. It is easy to check that
\[ P_k(q^n) = \prod_{d=1}^k \Big(\frac{1-q^{n-d+1}}{1-q^d}\Big) = \binom{n}{k}_q. \]
Now, suppose that $f,g: \N \longrightarrow \C$ which are both finitely supported. It suffices to show that $f$ vanishes identically. Taking $n$ larger than every element in $S = \supp(f)$ again, we find that
\begin{equation} \label{eq:g=binomq} g(n) = \sum_{k \in S} \binom{n-1}{k-1}_q f(k) = \sum_{k \in S} f(k) P_{k-1}(q^{n-1}). \end{equation}
Since $g(n) = 0$ for sufficiently large $n$, this implies that the polynomial
\[ \sum_{k \in S} f(k)P_{k-1}(u) \]
has infinitely many zeros as a polynomial in $u$. Since the only such polynomial is identically zero and since the polynomials $P_{k-1}$ are linearly independent, it follows that $S$ must be empty, as desired.
\end{proof}

Deducing the theorem from these propositions is immediate; it follows from using the information we have about the incidence coefficients.

We note here that the proofs of the three positive results enshrined in Theorem~\ref{thm:reduced} seem to have very different character and it seems unclear if there is a more poset-theoretic proof of it. I find it conceivable that in this setting, the analogue of Goh's conjecture (Conjectures~\ref{conj: goh}~and~\ref{conj: gohl}) may actually be correct even without the lattice assumption. That is, I think it may be the case that $P$ has property $\cR$ precisely when $\mu_P \in \sR$ is not finitely supported as a function on the natural numbers. Implicit in this assertion is a belief that the counterexample in Theorem~\ref{thm:gohcounterexample} has no typing; I did not, however, try to prove this.
\bibliography{poset}

@article {pollack,
    AUTHOR = {Pollack, Paul},
     TITLE = {The {M}\"obius transform and the infinitude of primes},
   JOURNAL = {Elem. Math.},
  FJOURNAL = {Elemente der Mathematik},
    VOLUME = {66},
      YEAR = {2011},
    NUMBER = {3},
     PAGES = {118--120},
      ISSN = {0013-6018,1420-8962},
   MRCLASS = {11A41},
  MRNUMBER = {2824427},
       DOI = {10.4171/EM/178},
       URL = {https://doi.org/10.4171/EM/178},
}

@article {pollacksanna,
    AUTHOR = {Pollack, Paul and Sanna, Carlo},
     TITLE = {Uncertainty principles connected with the {M}\"obius inversion
              formula},
   JOURNAL = {Bull. Aust. Math. Soc.},
  FJOURNAL = {Bulletin of the Australian Mathematical Society},
    VOLUME = {88},
      YEAR = {2013},
    NUMBER = {3},
     PAGES = {460--472},
      ISSN = {0004-9727,1755-1633},
   MRCLASS = {11A25 (11N37)},
  MRNUMBER = {3189296},
MRREVIEWER = {Giuseppe\ Melfi},
       DOI = {10.1017/S0004972712001128},
       URL = {https://doi.org/10.1017/S0004972712001128},
}

@article {sanna,
    AUTHOR = {Sanna, Carlo},
     TITLE = {On the asymptotic density of the support of a {D}irichlet
              convolution},
   JOURNAL = {J. Number Theory},
  FJOURNAL = {Journal of Number Theory},
    VOLUME = {134},
      YEAR = {2014},
     PAGES = {1--12},
      ISSN = {0022-314X,1096-1658},
   MRCLASS = {11A25 (11N37)},
  MRNUMBER = {3111554},
MRREVIEWER = {Guoyou\ Qian},
       DOI = {10.1016/j.jnt.2013.07.012},
       URL = {https://doi.org/10.1016/j.jnt.2013.07.012},
}

@article {goh,
    AUTHOR = {Goh, Marcel K.},
     TITLE = {An uncertainty principle for {M}\"obius inversion on posets},
   JOURNAL = {Contrib. Discrete Math.},
  FJOURNAL = {Contributions to Discrete Mathematics},
    VOLUME = {21},
      YEAR = {2026},
    NUMBER = {1},
     PAGES = {142--148},
      ISSN = {1715-0868},
   MRCLASS = {06A07},
  MRNUMBER = {5039869},
       DOI = {10.55016/6ej3y980},
       URL = {https://doi.org/10.55016/6ej3y980},
}

@article {rota,
    AUTHOR = {Rota, Gian-Carlo},
     TITLE = {On the foundations of combinatorial theory. {I}. {T}heory of
              {M}\"obius functions},
   JOURNAL = {Z. Wahrscheinlichkeitstheorie und Verw. Gebiete},
  FJOURNAL = {Zeitschrift f\"ur Wahrscheinlichkeitstheorie und Verwandte
              Gebiete},
    VOLUME = {2},
      YEAR = {1964},
     PAGES = {340--368},
   MRCLASS = {05.05},
  MRNUMBER = {174487},
MRREVIEWER = {M.\ A.\ Harrison},
       DOI = {10.1007/BF00531932},
       URL = {https://doi.org/10.1007/BF00531932},
}

@inproceedings {rota6reduced,
    AUTHOR = {Doubilet, Peter and Rota, Gian-Carlo and Stanley, Richard},
     TITLE = {On the foundations of combinatorial theory. {VI}. {T}he idea
              of generating function},
 BOOKTITLE = {Proceedings of the {S}ixth {B}erkeley {S}ymposium on
              {M}athematical {S}tatistics and {P}robability ({U}niv.
              {C}alifornia, {B}erkeley, {C}alif., 1970/1971), {V}ol. {II}:
              {P}robability theory},
     PAGES = {267--318},
 PUBLISHER = {Univ. California Press, Berkeley, CA},
      YEAR = {1972},
   MRCLASS = {05A15},
  MRNUMBER = {403987},
MRREVIEWER = {A.\ G.\ Law},
}

@article {rota4fq,
    AUTHOR = {Goldman, Jay and Rota, Gian-Carlo},
     TITLE = {On the foundations of combinatorial theory. {IV}. {F}inite
              vector spaces and {E}ulerian generating functions},
   JOURNAL = {Studies in Appl. Math.},
  FJOURNAL = {Studies in Applied Mathematics},
    VOLUME = {49},
      YEAR = {1970},
     PAGES = {239--258},
      ISSN = {0022-2526,1467-9590},
   MRCLASS = {05.10},
  MRNUMBER = {265181},
MRREVIEWER = {L.\ D.\ Baumert},
       DOI = {10.1002/sapm1970493239},
       URL = {https://doi.org/10.1002/sapm1970493239},
}

@book {stanley,
    AUTHOR = {Stanley, Richard P.},
     TITLE = {Enumerative combinatorics. {V}ol. 1},
    SERIES = {Cambridge Studies in Advanced Mathematics},
    VOLUME = {49},
      NOTE = {With a foreword by Gian-Carlo Rota,
              Corrected reprint of the 1986 original},
 PUBLISHER = {Cambridge University Press, Cambridge},
      YEAR = {1997},
     PAGES = {xii+325},
      ISBN = {0-521-55309-1; 0-521-66351-2},
   MRCLASS = {05-02 (05A15 06-02 11-02)},
  MRNUMBER = {1442260},
MRREVIEWER = {Wayne\ M.\ Dymacek},
       DOI = {10.1017/CBO9780511805967},
       URL = {https://doi.org/10.1017/CBO9780511805967},
}
\bibliographystyle{plain}

\end{document}